\documentclass[a4paper,11pt,reqno]{amsart}

\usepackage{amsmath}
\usepackage[T1]{fontenc}
\usepackage{amssymb}
\usepackage{amsthm}
\usepackage{color}
\usepackage[lmargin=2.5 cm,rmargin=2.5 cm,tmargin=3.5cm,bmargin=2.5cm,paper=a4paper]{geometry}

\newtheorem{thm}{Theorem}

\newtheorem{lemma}[thm]{Lemma}
\newtheorem{cor}[thm]{Corollary}
\theoremstyle{definition}

\let\epsilon=\varepsilon
\let\phi=\varphi

\def\ed{\,\mathrm{d}}

\newcommand{\R}{\mathbb R}
\newcommand{\N}{\mathbb N}
\newcommand{\Ab}{\mathbf{A}}
\newcommand{\Bb}{\mathbf{B}}

\title[Schr\"odinger Operators in exterior domains]{On the essential spectrum of magnetic Schr\"odinger operators in exterior domains}
\author{Ayman Kachmar}
\address[A. Kachmar]{Lebanese University, Department of Mathematics, Hadath,
Lebanon\newline Lebanese International University, School of Arts and
Sciences, Beirut, Lebanon}
\email[A. Kachmar]{ayman.kashmar@liu.edu.lb}
\author{Mikael Persson}
\address[M. Persson]{Centre for Mathematical Sciences, Box 118, SE-22100, Lund, Sweden.}
\email[M. Persson]{mickep@maths.lth.se}

\begin{document}

\begin{abstract}
We establish equality between the essential spectrum of the
Schr\"odinger operator with   magnetic field in the  exterior of a
compact arbitrary dimensional domain and that  of the operator
defined in all the space, and  discuss applications of this
equality.
\par \vspace*{10pt} \noindent {\sc R\'esum\'e.} {\textbf{\textit{Sur le spectre
de l'op\'erateur de Schr\"odinger avec un champ magn\'etique dans un
domaine exterieur.}}} On \'etablit une egalit\'e entre le spectre
essentiel de l'op\'erateur de Schr\"odinger avec un champ
magn\'etique dans un domaine exetrieur et celui de l'op\'erateur
dans tout l'espace. On discut des applications de cet egalit\'e.
\end{abstract}

\maketitle

\section{Introduction}

Magnetic Schr\"odinger operators in domains with boundaries appear
in several areas of physics, one can mention the Ginzburg-Landau
theory of superconductors, the theory of Bose-Einstein condensates,
Fermi-gases and the study of edge states in Quantum mechanics. We
refer the reader to \cite{AfHe, CFFH, FH} for details and additional
references on the subject. From the point of view of spectral
theory, the presence of boundaries has an effect similar to that of
perturbing the magnetic Schr\"odinger operator by an electric
potential. If we focus at present on two dimensional domains and
constant magnetic fields, we observe in both cases (exterior domain
and electric potential), that the essential spectrum consists of the
Landau levels and the discrete spectrum form clusters of eigenvalues
around the Landau levels. Several papers are devoted to the study of
different aspects of these clusters of eigenvalues in domains with
or without boundaries. In the semi-classical context, we can  can
cite \cite{CFFH, FK, Fr, HM, Krmp, Kjmp}, while \cite{P, puro}
contain results about accumulation  of eigenvalues in a non
semi-classical limit.

Consider a {\bf compact} and {\bf simply connected} domain
$K\subset\R^d$ with a {\bf smooth} $C^\infty$ boundary. Denote by
$\Omega=\R^d\setminus K$. Given a function $\gamma\in
L^\infty(\partial\Omega)$ and a vector potential $A\in
C^1(\R^d;\R^d)$, we define the Schr\"odinger operator
$L_{\Omega,\Bb}^\gamma$ with domain $D(L_{\Omega,\Bb}^\gamma)$ as
follows,
\begin{equation}\label{D-L}
D(L_{\Omega,\Bb}^\gamma)=\bigl\{u\in
L^2(\Omega)~:~(\nabla-i\Ab)^ju\in L^2(\Omega),~j=1,2;~
\nu\cdot(\nabla-i\Ab)u+\gamma u=0~{\rm
on~}\partial\Omega\bigr\}\,,
\end{equation}
\begin{equation}\label{Op-L}
\forall~u\in
D(L_{\Omega,\Bb}^\gamma)\,,\quad
L_{\Omega,\Bb}^\gamma u=-(\nabla-i\Ab)^2u\,.
\end{equation}
The vector $\nu$ is the unit {\it outward} normal vector of the
boundary $\partial\Omega$. The magnetic field $\Bb$ is identified by
an antisymmetric matrix $(b_{k,j})_{1\leq k,j\leq d}$ whose entries
are defined by the components $(a_j)$ of $\Ab$ as follows,
$b_{k,j}=\partial_{x_j}a_k-\partial_{x_k}a_j$. Associated to the
operator $L_{\Omega,\Bb}^\gamma$ is the quadratic form,
\begin{equation}\label{QF-q}
q_{\Omega,\Bb}^\gamma(u)=\int_\Omega|(\nabla-i\Ab)u|^2\ed
x+\int_{\partial\Omega}\gamma|u|^2\ed S\,,\quad u\in H^1_\Ab(\Omega)\,,
\end{equation}
where the space $H^1_{\Ab}(\Omega)= \{u\in
L^2(\Omega)~:~(\nabla-i\Ab)u\in L^2(\Omega)\}$ is the form domain of
$q_{\Omega,\Bb}^\gamma$.

Since the function $\gamma\in L^\infty(\partial\Omega)$, the
operator $L_{\Omega,\Bb}^\gamma$ is semi-bounded from below and its
associated quadratic form is closed. Freidrich's theorem   tells us
that $L_{\Omega,\Bb}^\gamma$ is self-adjoint  in $L^2(\Omega)$.

We introduce the magnetic Schr\"odinger operator $L_\Bb$ in
$L^2(\R^d)$ with magnetic field $\Bb$ as follows. The domain of the
operator is $ D(L_\Bb)=\{u\in L^2(\R^d)~:~(\nabla-i\Ab)^ju\in
L^2(\R^d),~j=1,2\}$,  and the action of the operator on its domain
is as follows,
\begin{equation}\label{op-L}
L_\Bb u=-(\nabla-i\Ab)^2u\,,\quad ({\rm
in}~L^2(\R^d))\,.\end{equation}

In this note, we  establish the following result and discuss consequences of it.
\begin{thm}\label{thm-KP}
The essential spectrum of the operator $L_{\Omega,\Bb}^\gamma$ is
the same as that of $L_\Bb$.
\end{thm}

Earlier versions of Theorem~\ref{thm-KP} are already proven for
two-dimensional domains \cite{P, puro} under different boundary
conditions and for constant magnetic fields only.
Theorem~\ref{thm-KP} remains true for the magnetic Schr\"odinger
operator with Dirichlet boundary condition (that is when replacing
the Robin condition in \eqref{D-L} by the condition $u=0$ on
$\partial\Omega$). The proof is exactly the same as the one we
present here.

\section{Proof of Theorem~\ref{thm-KP}}\label{Sec-proof}

We denote by $\Gamma$ the common boundary of $\Omega$ and $K$ and
define the following operator on $\Gamma$,
\begin{equation}\label{eq-normalder}
\partial_\Gamma u=\partial_Nu+\gamma\,u=\nu\cdot(\nabla-ib\Ab)u+\gamma\,u\,,
\end{equation}
where $\nu$ is the unit {\it outward} normal vector to the boundary
of $\Omega$.

We have introduced the operator $L_{\Omega,\Bb}^\gamma$ with
quadratic form $q_{\Omega,\Bb}^\gamma$ from \eqref{QF-q}. We will
use also the corresponding operator in $K$, namely
$L_{K,\Bb}^{-\gamma}$. Since the quadratic forms
$q_{\Omega,\Bb}^\gamma$ and $q_{\Omega,\Bb}^{-\gamma}$ are
semi-bounded (see \eqref{QF-q}), we get up to a shift by a positive
constat that they are strictly positive. Thus we assume, the
hypothesis:
\begin{itemize}
\item[(H1)] The operators $L_\Bb$, $L_{\Omega,\Bb}^\gamma$ and
$L_{K,\Bb}^{-\gamma}$  are invertible.
\end{itemize}

Since $\Omega$ and $K$ are complementary, the Hilbert space
$L^2(\R^d)$ is decomposed as the direct sum $L^2(\Omega)\oplus
L^2(K)$ in the sense that any function $u\in L^2(\R^d)$ can be
represented as $u_\Omega\oplus u_K$ where $u_\Omega$ and $u_K$ are
the restrictions of $u$ to $\Omega$ and $K$ respectively. Notice
that, for all $u=u_\Omega\oplus u_K\in L^2(\R^d)$ such that
$u_\Omega\in D(L_{\Omega,\Bb})$ and $u_K\in D(L_{K,\Bb})$, then
$\partial_\Gamma u_\Omega=\partial_\Gamma u_K=0$, where
$\partial_\Gamma$ is the trace operator from \eqref{eq-normalder}.

We can  extend the operator
$L_{\Omega,\Bb}^\gamma$ in $L^2(\Omega)$ to an operator $\widetilde L$ in
$L^2(\R^d)$. Actually, let $\widetilde L=L_{\Omega,\Bb}^\gamma\oplus
L_{K,\Bb}^{-\gamma}$ in $D(L_{\Omega,\Bb}^\gamma)\oplus D(L_{K,\Bb}^{-\gamma})\subset
L^2(\R^d)$. More precisely, $\widetilde L$ is the self-adjoint
extension associated with the quadratic form
\begin{equation}\label{eq-QF-q}
\widetilde q(u)=q_{\Omega,\Bb}^{\gamma}(u_\Omega)+q_{K,\Bb}^{-\gamma}(u_K)\,,\quad
u=u_\Omega\oplus u_K\in L^2(\R^d)\,.
\end{equation}
By the hypothesis (H1), we may speak of the resolvent $\widetilde
R=\widetilde L^{-1}$ of $\widetilde L$.  Since $\sigma(\widetilde
L)=\sigma(L_{\Omega,\Bb}^{\gamma})\cup\sigma(L_{K,\Bb}^{-\gamma})$
and $L_{K,\Bb}^{-\gamma}$  has a compact resolvent, then we get the
following lemma.
\begin{lemma}\label{lem:esssp}
With $\widetilde L$, $\widetilde R$ and $L_{\Omega,\Bb}^\gamma$ defined
as above, it holds true that:
\begin{enumerate}
\item $\sigma_{\text{ess}}(L_{\Omega,\Bb}^\gamma)=\sigma_{\text{ess}}(\widetilde L)$.
\item $\lambda\in\sigma_{\text{ess}}(\widetilde R)\setminus\{0\}$ if
and only if $\lambda\not=0$ and
$\lambda^{-1}\in\sigma_{{\rm ess}}(L_{\Omega,\Bb}^\gamma)$.
\end{enumerate}
\end{lemma}

In the next lemma, we observe that the operator
$L_{\Omega,\Bb}^{\gamma}$ can be viewed as a compact perturbation of
the magnetic Schr\"odinger operator $L_{\Bb}$ in $L^2(\R^d)$
introduced in \eqref{op-L}.

\begin{lemma}\label{lem:V}
The operator $V=\widetilde L^{-1}-L^{-1}_{\Bb}$ is compact.
Moreover, for all $f,g\in L^2(\R^d)$,
\begin{equation}\label{eq-V}
\langle f,Vg\rangle_{L^2(\R^d)}=\int_{\Gamma}\partial_\Gamma
u\cdot\overline{(v_\Omega-v_K)}\ed S\,,\end{equation} where
$u=L_\Bb^{-1}f$ and $v=\widetilde L^{-1} g$.
\end{lemma}
\begin{proof}
Since $f=L_\Bb u$
and $g=\widetilde L v= L_{\Omega,\Bb}^{\gamma} v_\Omega\oplus L_{K,\Bb}^{-\gamma}
v_K$, it follows that,
\begin{displaymath}
\langle f,Vg\rangle_{L^2(\R^d)}=\int_\Omega L_\Bb u\cdot \overline{v_\Omega}\ed
x+\int_KL_\Bb u\cdot\overline{v_K}\ed x-\int_\Omega
u\cdot\overline{L_{\Omega,\Bb}^{\gamma} v_\Omega}\ed x-\int_K
u\cdot\overline{L_{K,\Bb}^{-\gamma}v_K}\ed x\,.
\end{displaymath}
The identity in \eqref{eq-V} then
follows by integration by parts and by using the boundary conditions
$\partial_\Gamma v_\Omega=\partial_\Gamma v_K=0$.

As we will show below, compactness of the trace operators together with
\eqref{eq-V} give us compactness of the operator $V$.
Let $(g_n)$ be a sequence in $L^2(\R^d)$ that converges weakly to $0$.
We will prove that $(Vg_n)$ converges strongly in $L^2(\R^d)$. We define,
$$u^{(n)}=L_\Bb^{-1}Vg_n\,,\quad v^{(n)}=\widetilde L^{-1}g_n\,.$$
Let $U\subset\R^d$ be an open and bounded set that contains the
common boundary $\Gamma$ of $\Omega$ and $K$. We claim that there
exists a positive constant $C$ such that,
\begin{equation}\label{eq-bnd}
\forall~n\in\mathbb N\,,\quad
\|u^{(n)}\|_{H^2(\Omega)}+\|v^{(n)}_\Omega\|_{H^2(\Omega\cap U)}+\|v^{(n)}_K\|_{H^2(K\cap U)}\leq C\,.
\end{equation}
Once the estimate in \eqref{eq-bnd} is established, we get compactness of the operator $V$ as follows. Since the embeddings of $H^2(U\cap\Omega)$ and $H^2(\Omega\cap K)$ in $L^2(\Gamma)$ are compact, we get that,
$$\|v^{(n)}_\Omega\|_{L^2(\Gamma)}+ \|v^{(n)}_K\|_{L^2(K)}\to0\quad {\rm as}~n\to\infty\,.$$
Also, the trace theorem yields that $\|\partial_\Gamma
u^{(n)}\|_{L^2(\Gamma)}$ is bounded. Now, we may use \eqref{eq-V} with $f=Vg_n$, $g=g_n$ and deduce that,
$$\|Vg_n\|_{L^2(\R^d)}^2 =\int_{\Gamma}\partial_\Gamma
u^{(n)}\cdot\overline{(v^{(n)}_\Omega-v^{(n)}_K)}\ed S\to0\quad{\rm
as}~n\to\infty\,,$$ thereby establishing compactness of $V$. To
finish the proof of  Lemma~\ref{lem:V}, we need to prove the claim
in \eqref{eq-bnd}. Since $Vg_n$ is in $L^2(\R^d)$ we get by
definition of $L_\Bb^{-1}$ that $L_\Bb u^{(n)}=Vg_n$. As a
consequence, elliptic $L^2$-estimates yield boundedness of $u^{(n)}$
in $H^2(U)$. In a similar way we obtain boundedness of
$v^{(n)}_\Omega$ and $v^{(n)}_K$ in $H^2$. Actually, it holds true
that,
$$L_{\Bb,\Omega}v^{(n)}_\Omega=g_n \quad {\rm in}~\Omega\,,\quad L_{\Bb,K}v^{(n)}_K=g_n \quad {\rm in}~L^2(K)\,,$$
together with the boundary conditions $\partial_\Gamma
v^{(n)}_\Omega=0$ and $\partial_\Gamma v^{(n)}_K=0$. Boundedness of
$v^{(n)}_\Omega$ and $v^{(n)}_K$ in $H^2$ then result from elliptic
$L^2$-estimates (up to the boundary).
\end{proof}

\begin{proof}[Proof of Theorem~\ref{thm-KP}]
As corollary of Lemma~\ref{lem:V} and Weyl's theorem, we get that
$L_\Bb$ and $\widetilde L$ have the same essential spectrum. Consequently, Lemma~\ref{lem:esssp} tells us that Theorem~\ref{thm-KP} is true.
\end{proof}

\section{Applications of Theorem~\ref{thm-KP}}\label{sec-applications}

The spectrum of the operator $L_\Bb$ is studied in several papers,
see \cite{HM-88, Sh} and the references therein. Under the
assumptions made in Corollary~\ref{thm-HM} below, it is proved in
\cite[Thm.~1.5]{HM-88} that the essential spectrum of $L_\Bb$ is
exactly the union of spectra of all operators of the form
$L_{\Bb_\infty}$, where $\Bb_\infty$ is a cluster value of the
magnetic field $\Bb$ at $\infty$. The spectrum of $L_{\Bb_\infty}$
is either the interval $[\,|\Bb_\infty|,\infty)$ (if $\Bb_\infty=0$
or $d$ is odd) or the landau levels otherwise. Since
$L_{\Omega,\Bb}^\gamma$ has the same essential spectrum as $L_\Bb$,
we get the result in Corollary~\ref{thm-HM} below.

\begin{cor}\label{thm-HM}
Suppose $d=2,3$ and the magnetic field $\Bb\in C^{3}$ satisfies the
following condition,
\begin{equation}\label{eq-HM-con}
\sum_{1\leq |\alpha|\leq 3}\sum_{1\leq i,j\leq n}|D^\alpha b_{i,j}(x)|=\mathcal O\Big(|x|^{-\alpha}\Big)\,,\quad {\rm as~}|x|\to\infty\,,
\end{equation}
where $b_{i,j}$ are the components of $\Bb$ and $\alpha$ is a
positive real number. It holds true that:
\begin{enumerate}
\item If  $\displaystyle\liminf_{|x|\to\infty}|B(x)|=0$, then
$\displaystyle\sigma_{\rm ess} (L_{\Omega,\Bb}^\gamma)=
[0,\infty)$.
\item\label{st-sp}
If $d=2$, $\displaystyle\lim_{|x|\to\infty}|\Bb(x)|=b$ and
$b>0$,  then, $ \displaystyle\sigma_{\rm ess}
(L_{\Omega,\Bb}^\gamma)= \{(2n-1)b~:~n\in\N\} $.
\item If  $d=3$ and $\displaystyle\liminf_{|x|\to\infty}|\Bb(x)|=b$, then,
$\displaystyle\sigma_{\rm
ess}(L_{\Omega,\Bb}^\gamma)=[b,\infty)$.
\end{enumerate}
\end{cor}

The result of Corollary~\ref{thm-HM} remains true if the assumption
\eqref{eq-HM-con} is relaxed, by supposing that $\Bb(x)\to \mathbf
b$ as $|x|\to\infty$ and  $\Bb\in C^\infty$ (see
\cite[Theorems~2.2~\&~2.4]{Sh}).

The next corollary indicates situations where the spectrum of $L_{\Omega,\Bb}^\gamma$ is purely discrete.

\begin{cor}\label{thm-HM-*}
Suppose that there exists a non-negative integer $r$ such that $\Bb\in C^{r+1}(\R^d)$. Let $b_{k,j}$ be the components of $\Bb$.
If there exists a positive constant  $C$ such that,
$$\sum_{k,j}\sum_{\substack{\alpha\in\N^d\\|\alpha|= r+1}}|D^\alpha b_{k,j}(x)|\leq C\left(\sum_{k,j}\sum_{\substack{\alpha\in\N^d\\|\alpha|\leq r}}|D^\alpha b_{k,j}(x)|+1\right)\,,$$
and $\displaystyle\sum_{k,j}\sum_{\substack{\alpha\in\N^d\\|\alpha|\leq r}}|D^\alpha b_{k,j}(x)|\to\infty$ as $|x|\to\infty$, then  the
operator $L_{\Omega,\Bb}^\gamma$ has compact resolvent.
\end{cor}

Under the conditions in Corollary~\ref{thm-HM-*}, the operator
$L_\Bb$ has compact resolvent \cite[Corollaire~1.2]{HM-88}. As a
consequence of Lemma~\ref{lem:V}, we get that the operator
$L_{\Omega,\Bb}^\gamma$ has compact resolvent too, thereby proving
Corollary~\ref{thm-HM-*}.

\section*{Acknowledgements} AK is supported by a grant from Lebanese
University.

\end{document}